\documentclass[onefignum,onetabnum]{siamart171218}

\usepackage{amsmath}
\usepackage{amssymb}
\usepackage{enumerate}
\usepackage{mathrsfs}
\usepackage{color}
\usepackage{tikz}
\usepackage{mathtools}
\usepackage{breqn}

\def\R{{\mathbb{R}}}
\def\C{{\mathbb{C}}}

\newcommand{\defword}[1]{{\em #1}}

\newtheorem{example}[theorem]{Example}
\newtheorem{conjecture}[theorem]{Conjecture}

\usepackage{lipsum}
\usepackage{amsfonts}
\usepackage{graphicx}
\usepackage{epstopdf}
\usepackage{algorithmic}
\ifpdf
  \DeclareGraphicsExtensions{.eps,.pdf,.png,.jpg}
\else
  \DeclareGraphicsExtensions{.eps}
\fi

\newsiamremark{remark}{Remark}
\newsiamremark{hypothesis}{Hypothesis}
\crefname{hypothesis}{Hypothesis}{Hypotheses}
\newsiamthm{claim}{Claim}

\headers{Bistability of Sequestration Networks}{Xiaoxian Tang, and Jie Wang}

\title{Bistability of Sequestration Networks}

\author{Xiaoxian Tang\thanks{Department of Mathematics, Texas A\&M University, College Station, 77840 TX
  (\email{xiaoxian@math.tamu.edu}, \url{https://sites.google.com/site/rootclassification/}).}
\and Jie Wang\thanks{School of Mathematical Sciences, Peking University
  (\email{wangjie@math.pku.edu.cn}, \url{http://www.math.pku.edu.cn/teachers/wangjie/}).}
}

\usepackage{amsopn}
\DeclareMathOperator{\diag}{diag}

\ifpdf
\hypersetup{
  pdftitle={Bistability of Sequestration Networks},
  pdfauthor={Xiaoxian Tang, and Jie Wang}
}
\fi

\begin{document}

\maketitle

\begin{abstract}
 We  solve a  conjecture on multiple nondegenerate steady states, and prove bistability  for sequestration networks. More specifically, we prove that for any  odd number of species, and for any production factor, the fully open extension of a sequestration network admits three nondegenerate positive steady states, two of which are locally asymptotically stable. In addition, we provide a non-empty open set in the parameter space where a sequestration network admits bistability. 
\end{abstract}



\section{Introduction}\label{sec:intro}
Bistability is  an important problem to determine for given dynamical systems arising under mass-action kinetics from biochemical reaction networks \cite{CTF2006, CS2018, DK2008}.
Biologically, bistability is crucial for understanding basic phenomena such as decision-making process in cellular sigaling \cite{BF2001, FM1998, XF2003}.
Mathematically, identifying parameter values/regions for which a system exhibits two (or more) stable steady states is a challenging problem in computational real
algebraic geometry \cite{signs}. A necessary condition for bistability is multistationarity (the system has at least two distinct steady states). In practice, one way to experimentally observe bistability is
finding multistationarity. In many lucky cases, a witness for multistationarity gives at least three  distinct steady states, two of which are stable (see \cite{CS2018, OSTT2019}).
 Criterions for multistationarity have been widely studied, and many structured networks are well-understood (such as ``smallest" networks with a few species or reactions \cite{Joshi:Shiu:Multistationary},
(linearly) binomial networks \cite{DMST2019,TSS,SF}, conservative networks without boundary steady states \cite{CFMW} and  MESSI networks \cite{messi}).
However, given a general network,  it is not always true that
multistationarity guarantees  bistability.

Here we use algebraic methods to study both multistationarity and bistability for a family of important networks arising from biology: the fully open extensions  of sequestration networks
(see \cite{mss-review}, and variations in \cite{CF2005,SF1994}), i.e., sequestration networks with
all inflow and outflow reactions:
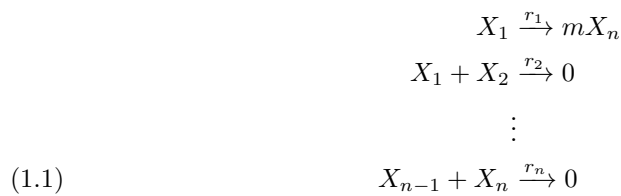
\begin{align}\label{eq:Kmn} \notag
X_1 &\xrightarrow{r_1} mX_n  \\ \notag
X_1 + X_2 &\xrightarrow{r_2} 0 \\ \notag
&\vdots \\
X_{n-1} + X_n &\xrightarrow{r_n} 0
\end{align}
\begin{align}\label{eq:inout}
X_i \xrightarrow{r_{n+i}} 0,\;\;\;\;  0\xrightarrow{r_{2n+i}} X_i, \;\;\;\; i=1,\ldots,n.
\end{align}
We are the first to prove  the following results.
\begin{enumerate}[(I)]
  	\item For any production factor $m\geq 2$, and for any odd order $n\geq 3$, the fully open extension  of sequestration network admits three nondegenerate steady states (Theorem~\ref{thm:mss}).
	\item  For any production factor $m\geq 2$, and for any odd order $n\geq 3$, the fully open extension  of sequestration network  admits bistability (Theorem~\ref{thm:bistability}).
	\item For any production factor $m\geq 2$, and for any odd order $n\geq 3$,  we provide an open region in the parameter space where the fully open extension  of sequestration network  admits bistability (Theorem~\ref{thm:region}).
\end{enumerate}

The fully open extensions of sequestration networks were first introduced in \cite{mss-review}, which were motivated by biochemical networks studied in \cite{CF2005, SF1994}. 
 Our main result (I) solves Conjecture  6.10 proposed in
\cite{mss-review} (see \cite[Conjecture 2.10]{FSW}). There are many well-known criteria for multistationarity by applying positive parametrization (e.g., \cite{jmp2018, TG}) and examining
the sign change of determinant of the Jacobian matrix (e.g., \cite{BP,CFMW,ME3,DMST2019,Feliu-inj,signs,ShinarFeinberg2012,WiufFeliu_powerlaw}).
Under some assumptions, one of these results \cite[Theorem 1]{CFMW} (or \cite[Theorem 3.12]{DMST2019}), proved by the Brouwer degree theory, guarantees an odd number of steady states when a network exhibits
multistationarity. But in general  there was no proof showing at least three of these steady states are nondegenerate.
Here, we use a strong algebraic technique to construct three nondegenerate steady states for sequestration networks $\widetilde{K}_{m,n}$ (see Lemma \ref{lm}, Lemma \ref{lm:third}, and Theorem \ref{thm:mss}).

A standard algebraic tool for studying stability is the Routh-Hurwitz criterion (see \cite{HTX2015}), or alternatively the Li\'enard-Chipart criterion (see \cite{datta1978}). 
Using these criteria, one examines  the positivity of some
gigantic determinants, which is computationally challenging (e.g., \cite{OSTT2019}).  Here, we discover a nice structure of the Jacobian matrices of  $\widetilde{K}_{m,n}$ at two of those three nondegenerate
steady states we constructed; specifically, they
are similar to diagonally dominant matrices. So, we are able to use the Gershgorin circle theorem to conclude stability  (see Lemma \ref{lm:similar3}, Lemma \ref{lm:similar5}, and Theorem \ref{thm:bistability}). We remark that the Gershgorin circle theorem can be used to study stability for more general reaction networks (see Theorem \ref{thm:matrix2}).
Also, we derive an open region in the parameter space for bistability, which is described by a set of positive solutions of finitely many polynomial inequalities in terms of rate constants (see Theorem \ref{thm:region}). We provide a procedure for computing a witness based on these inequalities and the proofs of Theorem \ref{thm:bistability}.

Finally, our work is related to the following open questions:
If a network admits multiple positive steady states, does this guarantee that  the network admits multiple nondegenerate positive steady
states?  (See Nondegeneracy
Conjecture~\cite{Joshi:Shiu:Multistationary, shiu-dewolff}.)  If a network admits multiple nondegenerate positive steady states, under which condition does the network admit multiple stable positive steady states?

The rest of this paper is organized as follows.
In Section~\ref{sec:pre}, we introduce
mass-action kinetics systems  arising  from reaction networks.
In Section~\ref{sec:matrix}, we introduce an algebraic criterion for stability (Theorem \ref{thm:matrix2}), which is deduced by the classical Gershgorin circle theorem (Theorem \ref{thm:matrix1}).
In Section~\ref{sec:seqnet}, we recall a family of sequestration networks defined in \cite{mss-review}, and present our main results (I--III) (Theorems \ref{thm:mss}, \ref{thm:bistability} and \ref{thm:region}). 
In Section \ref{sec: proof}, we prove the main results in details. We end with a summary In Section \ref{con}.

\section{Reaction networks}\label{sec:pre}


In this section, we briefly recall the standard notions and definitions on reaction networks, see \cite{CFMW, DMST2019} for more details.
A \defword{reaction network} $G$  (or {\em network} for short) consists of a set of $s$ species $\{X_1, X_2, \ldots, X_s\}$ and a set of $m$ reactions:
\[
\alpha_{1j}X_1 +
\alpha_{2j}X_2 +  \dots +
\alpha_{sj}X_s
~ \xrightarrow{r_j} ~
\beta_{1j}X_1 +
\beta_{2j}X_2 +  \dots +
\beta_{sj}X_s,
 \;
    {\rm for}~
	j=1,2, \ldots, m,
\]
where all $\alpha_{ij}$ and $\beta_{ij}$ are non-negative integers. We call the $s\times m$ matrix with
$(i, j)$-entry equal to $\beta_{ij}-\alpha_{ij}$ the
\defword{stoichiometric matrix} of
$G$,
denoted by $N$.
We call the image of $N$
the \defword{stoichiometric subspace}, denoted by $S$.

We denote by $x_1, x_2, \ldots, x_s$ the concentrations of the species $X_1,X_2, \ldots, X_s$, respectively.
Under the assumption of {\em mass-action kinetics}, we describe how these concentrations change  in time by following system of ODEs:
\begin{equation}\label{sys}
\dot{x}~=~f(x)~:=~N\cdot \begin{pmatrix}
r_1 \, x_1^{\alpha_{11}}
		x_2^{\alpha_{21}}
		\cdots x_s^{\alpha_{s1}} \\
r_2 \, x_1^{\alpha_{12}}
		x_2^{\alpha_{22}}
		\cdots x_s^{\alpha_{s2}} \\
		\vdots \\
r_m \, x_1^{\alpha_{1m}}
		x_2^{\alpha_{2m}}
		\cdots x_s^{\alpha_{sm}} \\
\end{pmatrix}~,
\end{equation}
where $x$ denotes the vector $(x_1, x_2, \ldots, x_s)$,
and each $r_j \in \mathbb R_{>0}$ is called a \defword{reaction rate constant}.
 By considering the rate constants as a vector $r=(r_1, r_2, \dots, r_m)$, we have polynomials $f_{i} \in \mathbb Q[r, x]$, for $i=1,2, \dots, s$.

A \defword{positive steady state} (or, simply \defword{steady state})\footnote{Usually, a steady state is defined as a non-negative vector $x\in {\mathbb R}_{\geq 0}^s$.
In our setting, we do not
consider boundary steady states (i.e., steady states with zero coordinates). So all steady states in our context are positive.} of~\eqref{sys} is a concentration-vector
$x^* \in \mathbb{R}_{> 0}^s$ at which $f(x)$ on the
right-hand side of the
ODEs~\eqref{sys}  vanishes, i.e., $f(x^*) =0$.
We say a steady state $x^*$ is \defword{nondegenerate} if the image of
${\rm Jac}(f) (x^*)|_{S}$
is equal to the stoichiometric subspace $S$,
where ${\rm Jac}(f)(x^*)$ denotes the Jacobian matrix of $f$, with respect to $x$, at $x^*$.
Notice that when the stoichiometric matrix $N$ is full rank, a steady state $x^*$ is nondegenerate if
  ${\rm Jac}(f)(x^*)$ is full rank.
A steady state $x^*$ is said to be \defword{Liapunov stable} if for any $\epsilon>0$ and for any $t_0>0$, there exists $\delta>0$ such that
$\parallel x(t_0)-x^*\parallel<\delta$ implies $\parallel x(t)-x^*\parallel<\epsilon$ for any $t\geq t_0$.
A steady state $x^*$ is said to be \defword{locally asymtotically stable} if it is Liapunov stable, and there exists $\delta>0$  such that
$\parallel x(t_0)-x^*\parallel<\delta$ implies $\lim_{t\rightarrow \infty} x(t)=x^*$.
It is well-known that a steady state $x^*$ is
locally asymtotically stable if all eigenvalues of ${\rm Jac}(f)(x^*)$ have negative real parts.

\section{A criterion for stability}\label{sec:matrix}
%
%


\begin{definition}\label{def:disc}
Let $A=(a_{ij})\in\R^{n\times n}$ be a matrix. For every $i=1,\ldots,n$, define the $i$-th {\em row Gershgorin disc} of $A$ in the complex plane as the set
$$R_i := \{z\in\C:~ |z-a_{ii}|\le\sum_{j\ne i}|a_{ij}|~\}.$$
Similarly, define the $i$-th {\em column Gershgorin disc} of $A$ in the complex plane as the set
$$C_i:=\{z\in \C: |z-a_{ii}|\le\sum_{j\ne i}|a_{ji}|\}.$$ 
\end{definition}


\begin{theorem}\label{thm:matrix1}\cite[Gershgorin circle theorem]{gerschgorin1931}
The eigenvalues of a  matrix $A=(a_{ij})\in\R^{n\times n}$ lie in the union of row Gershgorin discs $\cup_{i}^nR_i$, and also lie in the union of column Gershgorin discs $\cup_{i}^nC_i$.
\end{theorem}

\begin{definition}\label{def:dom}
Let $A=(a_{ij})\in\R^{n\times n}$ be a matrix. If for every $i=1, \ldots, n$,  $|a_{ii}|\ge\sum_{j\ne i}|a_{ij}|$ {\em (}or, $|a_{ii}|\ge\sum_{j\ne i}|a_{ji}|${\em )}, then $A$ is {\em row diagonally dominant} {\em (}or, {\em column diagonally dominant}{\em )}.
\end{definition}

For a diagonally dominant matrix, we have a simple sufficient condition for its stability by virtue of the Gershgorin circle theorem.
\begin{lemma}\label{lm:matrix2}
Let $A=(a_{ij})\in\R^{n\times n}$ be a (row or column) diagonally dominant matrix. If $a_{ii}<0$ for every $i=1, \ldots, n$, then every nonzero eigenvalue of $A$ has a negative real part.
\end{lemma}
\begin{proof}

Let $\lambda$ be a nonzero eigenvalue of $A$.
Denote respectively the real and imaginary parts of $\lambda$ by $Re(\lambda)$ and $Im(\lambda)$. Then $Re(\lambda)\neq 0$, or $Im(\lambda)\neq 0$.  Without loss of generality, assume $A$ is  row diagonally dominant. Note for any $i$, $a_{ii}<0$. So if $Re(\lambda)\geq 0$, then
for any $i$, we have
\[|\lambda-a_{ii}|=\sqrt{(Re(\lambda)-a_{ii})^2 +Im(\lambda)^2}>a_{ii} \ge\sum_{j\ne i}|a_{ij}|,\]
which is a contradiction to Theorem \ref{thm:matrix1}. Hence, we must have $Re(\lambda)<0$.
\end{proof}

\begin{theorem}\label{thm:matrix2}
If a matrix is similar to a (row or column) diagonally dominant matrix with negative diagonal entries, then all the nonzero eigenvalues have negative real parts.
\end{theorem}
\begin{proof}
The conclusion directly follows from Lemma \ref{lm:matrix2} and the fact  that similar matrices have the same eigenvalues.
\end{proof}

\section{Sequestration networks and main results}\label{sec:seqnet}

\subsection{Preliminary}\label{sec:spre}

In this section, we recall sequestration networks $K_{m, n}$ \cite[Definition 6.3]{mss-review} and their fully open extensions $\widetilde{K}_{m, n}$ \cite[Definition 2.3]{mss-review}.

\begin{definition}\label{def:seqnet}
For any integer $m\geq 1$, and for any integer $n\geq 2$, the {\em sequestration network} $K_{m, n}$ of order $n$ with production factor $m$ is defined to be the network \eqref{eq:Kmn}.
If we add into \eqref{eq:Kmn} all inflow reactions and outflow reactions \eqref{eq:inout},
then we obtain the {\em fully open extension} of $K_{m,n}$, denoted by $\widetilde{K}_{m,n}$.
\end{definition}

According to \eqref{sys}, the mass-action ODEs $\dot{x}=f(x)$  of $\widetilde{K}_{m,n}$ are given by:
\begin{equation}\label{eq:oldsystem}
\begin{cases}
f_1~=~-r_1x_1x_2-r_nx_1-r_{n+1}x_1+r_{2n+1},\\
f_i\,~=~-r_{i-1}x_{i-1}x_{i}-r_{i}x_{i}x_{i+1}-r_{n+i}x_{i}+r_{2n+i},
~\textrm{ for }~2\le i\le n-1,\\
f_n~=~-r_{n-1}x_{n-1}x_n+mr_{n}x_{1}-r_{2n}x_n+r_{3n}.
\end{cases}
\end{equation}
The Jacobian matrix of $f$ with respect to $x_1,\ldots,x_n$ below is simply denoted by $J$:
\begin{equation}\label{eq:jac}
{\tiny
\begin{bmatrix}
-r_1x_2-r_n-r_{n+1}&-r_1x_1&\cdots&0&0\\
-r_1x_2&-r_1x_1-r_2x_3-r_{n+2}&\cdots&\vdots&\vdots\\
0&-r_2x_3&\ddots&0&0\\
\vdots&0&\ddots&-r_{n-2}x_{n-2}&0\\
0&\vdots&\ddots&-r_{n-2}x_{n-2}-r_{n-1}x_n-r_{2n-1}&-r_{n-1}x_{n-1}\\
mr_n&0&\cdots&-r_{n-1}x_n&-r_{n-1}x_{n-1}-r_{2n}
\end{bmatrix}.
}
\end{equation}

\begin{definition}\label{def:mss}
The network $\widetilde{K}_{m,n}$ is \defword{multistationary} 
(respectively, \defword{bistable})
if, for some choice of positive rate-constant vector $r \in \mathbb{R}^{3n}_{>0}$, there exist  two or more positive steady states (respectively, locally asymptotically stable positive steady states) of~\eqref{eq:system}.
\end{definition}

It is known that  for any integers $m\geq 1$, and for any  integers $n\geq 2$,  $\widetilde{K}_{m,n}$ is multistationary if and only if $m>1$ and $n$ is odd \cite[Theorem 6.4]{mss-review}.
There is a conjecture that  {\bf  for any integers $m\geq 2$, and for odd integers $n\geq 3$,  $\widetilde{K}_{m,n}$ admits multiple nondegenerate steady states \cite[Conjecture 6.10]{mss-review}}.
Notice that the stoichiometric matrix $N$ of $\widetilde{K}_{m,n}$ is full rank (e.g., see \cite[Formula (4)]{FSW}), so the conjecture says
for some  rate-constant vector $r^* \in \mathbb{R}^{3n}_{>0}$, there exist at least two positive steady states $x^{(1)}$ and $x^{(2)}$ such that $\det J|_{r=r^*, x=x^{(i)}}\neq 0,i=1,2$.
For any integers $m\geq 2$, and for $n=3$, the conjecture was resolved in \cite[Theorem 4.5]{FSW}.
For $m=2, 3, 4, 5$, and for $n=5, 7, 9, 11$, the conjecture was proved in \cite[Theorem 5.1]{FSW}.

\subsection{Main results}\label{sec:main}

Our main results are a proof of \cite[Conjecture 6.10]{mss-review} (see Theorem \ref{thm:mss}) and a bistability result for sequestration networks (see Theorem \ref{thm:bistability}).
Also, we provide an open region in the parameter space where $\widetilde{K}_{m,n}$ admits bistability (see Theorem \ref{thm:region}).
The proofs of these results are given later in Section \ref{sec: proof}. 
A procedure for computing a witness for bistability is presented (see {\bf Procedure Witness}).
A concrete example of $\widetilde{K}_{m, n}$ with two  locally asymptotically stable steady states   is given (see Example \ref{example}),  which is not covered by \cite[Theorem 4.5]{FSW} or  \cite[Theorem 5.1]{FSW}

\begin{theorem}\label{thm:mss}
For any integer $m\geq 2$, if for any odd integer $n>3$, the rate-constant vector $(r_1, \ldots, r_{n}, r_{n+2}) \in {\mathbb R}^{n+1}_{>0}$ belongs to the open set determined by the following polynomial inequalities
\eqref{eq:nondegenerate}--\eqref{eq:sec3-eq12}
\begin{align}
(r_1+r_n)r_{n+2}~&\neq~(m-1)r_1r_n, \label{eq:nondegenerate}\\
r_{n-1}~&>~mr_n, \label{sec-eq0}\\
(m-1)r_1~&>~r_{n+2}, \label{eq:conie-eq11}\\
 (m-1)r_1(r_{i-1}+(-1)^imr_{n})~&>~(-1)^{i}m(r_1+r_{n})r_{n+2}, \;\;i=3,\ldots,n, \label{eq:conie-eq12}\\
r_1+r_{n+2}~&>~r_{n-2},\label{eq:sec3-eq11}\\
r_i~&>~r_{n-2},\;\;\;\;\;\;\;\;\;\;\;\;\;\;\;\;\;\; i=3,5,\ldots,n-4,\label{eq:sec3-eq12} 
\end{align}
or, if for $n=3$, the rate-constant vector $(r_1, r_2, r_3, r_5) \in {\mathbb R}^4_{>0}$ belongs to the open set determined by the polynomial inequalities \eqref{eq:nondegenerate}--\eqref{eq:conie-eq12},
then there exist rate constants $r_{n+1}, r_{n+3}, \ldots, r_{3n}>0$  such that $\widetilde{K}_{m,n}$ has three nondegenerate steady states.
Moreover,  the above  open set  in ${\mathbb R}_{>0}^{n+1}$ is non-empty.
\end{theorem}

\begin{remark}
As mentioned before, for $n=3$,   the original conjecture (\cite[Conjecture 6.10]{mss-review}) was already proved in \cite[Theorem 4.5]{FSW}.
However, we still provide a self-contained proof in Section \ref{sec: proof} because we need the construction of three nondegenerate steady states  shown in our proof to demonstrate the bistability result (see Theorem \ref{thm:bistability}).
\end{remark}

\begin{remark}
If we replace the condition \eqref{eq:nondegenerate} listed in Theorem \ref{thm:mss} with the inequality \eqref{eq:constable1} below, then we can conclude that there are two stable steady states among the three nondegenerate steady states stated in Theorem \ref{thm:mss}, where one of the two stable steady states is $(1, 1, \ldots, 1)$ (see Theorem \ref{thm:bistability}).
\end{remark}

\begin{theorem}[Bistability]\label{thm:bistability}
 For any integer $m\geq 2$, if for any odd integer $n>3$, the rate-constant vector $(r_1, \ldots, r_{n}, r_{n+2}) \in {\mathbb R}^{n+1}_{>0}$
belongs to the open set determined by the polynomial inequalities
\eqref{sec-eq0}--\eqref{eq:sec3-eq12} and the following inequality
\begin{align}
(r_1+r_n)r_{n+2}~&>~(m-1)r_1r_n, \label{eq:constable1}
\end{align}
or, if for $n=3$, the rate-constant vector $(r_1, r_2, r_3, r_5) \in {\mathbb R}^4_{>0}$
 belongs to the open set determined by the polynomial inequalities \eqref{sec-eq0}--\eqref{eq:conie-eq12} and \eqref{eq:constable1},
then there exist rate constants $r_{n+1}, r_{n+3}, \ldots, r_{3n}>0$  such that $\widetilde{K}_{m,n}$ has
two locally asymptotically stable steady states, and one of these two steady states is $(1, 1, \ldots, 1)$. Moreover,  the above  open set  in ${\mathbb R}_{>0}^{n+1}$ is non-empty.
\end{theorem}

 In Theorem \ref{thm:bistability}, it is obvious that  the set of positive solutions of the inequalities \eqref{sec-eq0}--\eqref{eq:constable1} for $n>3$ (or, the inequalities \eqref{sec-eq0}--\eqref{eq:conie-eq12} and \eqref{eq:constable1} for $n=3$) is an open set in ${\mathbb R}_{>0}^{n+1}$.
In order to make it more obvious to see the open set is non-empty, we provide Theorem \ref{thm:region}, which explicitly describe the positive solutions of  the inequalities stated in Theorem \ref{thm:bistability}.

\begin{theorem}\label{thm:region}
For any integer $m\geq 2$, and for $n=3$, the open set  in ${\mathbb R}_{>0}^{n+1}$ determined by  the inequalities by \eqref{sec-eq0}--\eqref{eq:conie-eq12} and \eqref{eq:constable1} in Theorem \ref{thm:bistability} is equivalent to the following set:
\begin{equation}\label{eq3:prop1}
(r_1, r_2, r_3, r_{5}) \in {\mathbb R}^{4}_{>0}:\begin{cases}
r_{5}~<~(m-1)r_1,\\
r_{2}~>~mr_3.
\end{cases}
\end{equation}
For any integer $m\geq 2$, and for any odd integer $n>3$,  the open set  in ${\mathbb R}_{>0}^{n+1}$ determined by the inequalities \eqref{sec-eq0}--\eqref{eq:constable1} in Theorem \ref{thm:bistability} is equivalent to the following set:
\begin{equation}\label{eq3:prop}
(r_1, \ldots, r_{n}, r_{n+1}) \in {\mathbb R}^{n+1}_{>0}:\begin{cases}
r_{n+2}~<~(m-1)r_1,\\
r_n~<~\frac{r_1r_{n+2}}{(m-1)r_1-r_{n+2}},\\
r_{n-1}~>~mr_n,\\
\frac{m((r_1+r_n)r_{n+2}-(m-1)r_1r_n)}{(m-1)r_1}~<~r_{n-2}<r_1+r_{n+2},\\
r_i~>~r_{n-2},\;\;\;\;\;\;\; \;\;\;\textrm{ for }i=3,5,\ldots,n-4. 
\end{cases}
\end{equation}
\end{theorem}

\begin{remark}\label{rmk:witness}
By the inequalities in \eqref{eq3:prop1} and \eqref{eq3:prop},
one can easily choose a rate-constant vector such that the conditions of  Theorem \ref{thm:bistability} are satisfied.
In fact, for any integer $m\geq 2$, if $n=3$, then
for any fixed $r_1, r_3>0$, there alway exist $r_5, r_2>0$ such that the two inequalities in \eqref{eq3:prop1} are satisfied.
If $n>3$, notice that the inequalities in \eqref{eq3:prop} have a ``triangular" shape. More specifically,
first, for any fixed $r_1>0$, one can always choose $r_{n+2}>0$ such that  the first inequality is satisfied.
Second, for  the chosen $r_1, r_{n+2}>0$ in the first step,  one can find  $r_n>0$ such that  the second inequality is satisfied.
Third, for the chosen $r_n>0$ in the second step, one can find $r_{n-1}>0$ such that  the third inequality is satisfied.
Similarly, we can find $r_{n-2}$ and $r_i$ for $i=3,5,\ldots,n-4$ by the last two inequalities (notice that in the fourth  inequality, there exists $r_{n-2}>0$ between the two numbers
$\frac{m((r_1+r_n)r_{n+2}-(m-1)r_1r_n)}{(m-1)r_1}$ and $r_1+r_{n+2}$ because
$\frac{m((r_1+r_n)r_{n+2}-(m-1)r_1r_n)}{(m-1)r_1}<r_1+r_{n+2}$ is implied by the first inequality $r_{n+2}~<~(m-1)r_1$).
Notice that $r_{2}, r_{4}, \ldots, r_{n-3}$ do not appear in the inequalities \eqref{eq3:prop}. We can choose any positive values for them.
 For instance, we give the following choices.

For $n=3$, we can choose $r_1=2, r_2=m+1, r_3=1$, and $r_5=m-1$ such that
the inequalities \eqref{sec-eq0}--\eqref{eq:conie-eq12} and \eqref{eq:constable1} in Theorem \ref{thm:bistability} are satisfied.

For any odd integer $n>3$, 
we can choose $r_1=2, r_2=r_4=\cdots=r_{n-3}=1, r_3=r_5=\cdots=r_{n-4}=m+1, r_{n-2}=m,  r_{n-1}=m+1, r_n=1$, and $r_{n+2}=m-1$ such that
the inequalities \eqref{sec-eq0}--\eqref{eq:constable1} in Theorem \ref{thm:bistability} are satisfied.
\end{remark}


Based on Theorem \ref{thm:region} and the proofs of Theorems \ref{thm:mss} and \ref{thm:bistability} (in Section \ref{sec: proof}), we provide a procedure ({\bf Procedure Witness}) for computing a witness for bistability.
Notice that Step 1 in the procedure below can be carried out according to Remark \ref{rmk:witness}.
We give a more concrete example later for $m=6$ and $n=5$; see Example \ref{example}.

{\bf Procedure Witness.}
\underline{Input.} $m\geq2$, and odd $n\geq 3$;
\underline{Output.} $r_1, \ldots, r_{3n}>0$ such that $\widetilde{K}_{m,n}$ is bistable.\\
\underline{Step 1.}
For $n=3$, find values for $r_1, r_2, r_3, r_{5}>0$ by  \eqref{eq3:prop1} such that the inequalities \eqref{sec-eq0}--\eqref{eq:conie-eq12} and \eqref{eq:constable1}  are satisfied.
For $n>3$, find values for  $r_1, \ldots, r_n, r_{n+2}>0$ by  \eqref{eq3:prop} such that the inequalities \eqref{sec-eq0}--\eqref{eq:constable1} are satisfied. \\
\underline{Step 2.}
Let $r_{n+1}=r_{n+3}=\ldots=r_{2n}=\epsilon>0$.\\
\underline{Step 3.}
 Compute values for $r_{2n+1}, \ldots, r_{3n}$ by the equalities:
 \begin{equation}\label{eq:coneq}
\begin{cases}
r_{2n+1}&=~r_1+r_n+r_{n+1},\\
r_{2n+i}&=~r_{i-1}+r_i+r_{n+i},\quad\textrm{ for }\quad2\le i\le n-1,\\
r_{3n}&=~r_{n-1}-mr_n+r_{2n}.
\end{cases}
\end{equation}
 \\
 \underline{Step 4.}
 Compute steady states of $\widetilde{K}_{m,n}$ and check their stability (for instance, by {\tt Mathematica}). If  $\widetilde{K}_{m,n}$ is bistable, then output $r_1, \ldots, r_{3n}$.
 Otherwise, go back to Step 2, make $\epsilon$ smaller and repeat Steps 2--4 until $\widetilde{K}_{m,n}$ is bistable.

\begin{example}\label{example}
We give a concrete example of $\widetilde{K}_{6,5}$ with two locally
asymptotically stable steady states. 
Let $r_1=2$, $r_2=r_5=1$, $r_3=6$, $r_4=7$, $r_7=5$, $r_6=r_8=r_9=r_{10}=0.006$, $r_{11}=3.006$, $r_{12}=8$, $r_{13}=7.006$, $r_{14}=13.006$, and $r_{15}=1.006$. Here, the values of rate constants $r_1, \ldots, r_5$ and $r_7$ are chosen by the method described in  Remark \ref{rmk:witness}, which satisfy the inequalities \eqref{sec-eq0}--\eqref{eq:constable1}.
By the proof of Theorem \ref{thm:bistability} (see Section \ref{sec:proofbi}), the values for $r_6, r_8, r_9$ and $r_{10}$ are chosen to be the same small number $0.006$. After we choose these values for
$r_1, \ldots, r_{10}$,
the values of $r_{11}, \ldots, r_{15}$ are computed by the equalities \eqref{eq:coneq}.
The specialized system $f$ in \eqref{eq:oldsystem} is given by
\begin{equation*}
\begin{cases}
f_1~=~-2x_1x_2-1.006x_1+3.006,\\
f_2~=~-2x_{1}x_{2}-x_{2}x_{3}-5x_{2}+8,\\
f_3~=~-x_{2}x_{3}-6x_{3}x_{4}-0.006x_{3}+7.006,\\
f_4~=~-6x_{3}x_{4}-7x_{4}x_{5}-0.006x_{4}+13.006,\\
f_5~=~-7x_{4}x_5+6x_{1}-0.006x_n+1.006.
\end{cases}
\end{equation*}
It can be verified by {\tt Maple} \cite{maple} that the above system $f=0$ has three positive solutions:
$$\hat x^{(1)}=(1,1,1,1,1),\quad\hat x^{(2)}\approx (1.69795,0.382186,12.5363,0.028445,54.5727)$$
and
$\hat x^{(3)}\approx (1.92826,0.276459,20.0808,0.0110718,150.601),$
where $\hat x^{(1)}$ and $\hat x^{(3)}$ are locally asymptotically stable.
Indeed, the Jacobian matrix at $\hat x^{(1)}$ has five negative eigenvalues, which are approximately
$$-19.7034, \;-9.28405, \;-6.17915,\; -2.78462,\; -0.07275,$$ and the Jacobian matrix at $\hat x^{(3)}$ has  five negative  eigenvalues, which are approximately
$$-1174.78,\; -29.2068,\; -1.49192,\; -0.151575,\; -0.00198971.$$

\end{example}

\begin{remark}\label{rmk:condition}
In Theorem \ref{thm:bistability},
if we replace the inequality \eqref{eq:constable1} with its opposite
\begin{align}\label{eq:constable2}
(r_1+r_n)r_{n+2}~&<~(m-1)r_1r_n,
\end{align}
one can still prove (in a similar way with the proof of Theorem \ref{thm:bistability}) that $\widetilde{K}_{m,n}$ admits two locally asymptotically stable steady states, and one of the two stable steady states is close to $(\delta_1, \ldots, \delta_n)$ given in \eqref{eq:2solutions} (Section \ref{sec: proof}).
For instance, for any integer $m\geq 2$, when $n=3$, we can choose $r_1=3,r_2=3m,r_3=2,r_5=m-1$ such that
the inequalities  \eqref{sec-eq0}--\eqref{eq:conie-eq12} and \eqref{eq:constable2} are satisfied,  and when $n>3$ is odd, we can choose $r_1=3, r_2=r_4=\cdots=r_{n-3}=m, r_3=r_5=\cdots=r_{n-4}=m+1, r_{n-2}=m, r_{n-1}=3m,r_n=2,r_{n+2}=m-1$ such that
the inequalities  \eqref{sec-eq0}--\eqref{eq:sec3-eq12} and \eqref{eq:constable2} are satisfied. We give another example to illustrate this case; see Example \ref{ex:ex2}.
\end{remark}

\begin{example}\label{ex:ex2}
Again, we consider the network $\widetilde{K}_{6,5}$.
Let $r_1=3$, $r_2=r_3=6$, $r_4=18$, $r_5=2$, $r_7=5$, $r_6=r_8=r_9=r_{10}=0.06$, $r_{11}=5.06$, $r_{12}=14$, $r_{13}=12.06$, $r_{14}=24.06$, and $r_{15}=6.06$. This time, these rate constants satisfy
the inequalities  \eqref{sec-eq0}--\eqref{eq:sec3-eq12} and \eqref{eq:constable2}.
It can be verified by {\bf Maple} that  there are three positive steady states:
$$\hat x^{(1)}=(1,1,1,1,1),\quad\hat x^{(2)}\approx (0.932124,1.12282,0.778704,1.44839,0.659961)$$
and
$\hat x^{(3)}\approx (1.68739,0.312906,5.77995,0.0248477,51.8643),$
where $\hat x^{(2)}$ and $\hat x^{(3)}$ are locally asymptotically stable. Indeed, the Jacobian matrix at $\hat x^{(2)}$ has five negative eigenvalues, which are approximately $$-40.2232, \;-20.8642, \;-7.79735,\; -7.20777,\; -0.0343658,$$ and the Jacobian matrix at $\hat x^{(3)}$ has  five negative  eigenvalues, which are approximately $$-968.734, \;-46.3232,\; -2.9517, \;-0.5785,\; -0.0443525.$$
\end{example}

\begin{remark}
Numerical experiments show that
if the conditions of Theorem \ref{thm:mss} are not satisfied, then it is possible for $\widetilde{K}_{m,n}$ to admit either one or three nondegenerate steady states.
We have never seen more than three nondegenerate steady states.
We always observe
bistability whenever three nondegenerate steady states are found.
So,  we propose Conjecture \ref{conj} below.
\end{remark}
\begin{conjecture}\label{conj}
For any integer $m\geq 2$,  and for any  integer $n\ge3$, the
maximum number of nondegerate steady states of $\widetilde{K}_{m,n}$  is three, and the network $\widetilde{K}_{m,n}$ is multistationary if and only if it is bistable.
\end{conjecture}

\section{Proofs of main results}\label{sec: proof}
The goal of this section is to prove Theorem \ref{thm:mss}, Theorem \ref{thm:bistability} and Theorem \ref{thm:region}.
Our first step is to apply the specializations of parameters \eqref{eq:coneq}\footnote{These specializations are inspired by the proof of \cite[Theorem 4.5]{FSW}.} to the network.
Substituting \eqref{eq:coneq} into the system $f$ \eqref{eq:oldsystem}, the system can be rewritten as
\begin{equation}\label{eq:system}
\begin{cases}
f_1=-r_1x_1x_2-r_nx_1-r_{n+1}x_1+r_1+r_n+r_{n+1},   \\
f_i\,=-r_{i-1}x_{i-1}x_{i}-r_{i}x_{i}x_{i+1}-r_{n+i}x_{i}+r_{i-1}+r_i+r_{n+i},
\textrm{ for }2\le i\le n-1,  \\
f_n=-r_{n-1}x_{n-1}x_n+mr_{n}x_{1}-r_{2n}x_n+r_{n-1}-mr_n+r_{2n}.
\end{cases}
\end{equation}
Note that by the equalities \eqref{eq:coneq}, $x^{(1)}:=(1, \ldots, 1)$ is always a positive solution to the system \eqref{eq:system}.
Note also that this substitution does not change the Jacobian matrix of $f$ with respect to $x$ since $r_{2n+1}, \ldots, r_{3n}$ are constant terms in \eqref{eq:oldsystem}.

 Under the equalities \eqref{eq:coneq}, we only need to find rate constants $r_1, \ldots, r_{2n}>0$ such that
the system $f=0$ in \eqref{eq:system} has three distinct simple positive solutions. Then by
\eqref{eq:coneq}, we can find positive values for rate constants $r_{2n+1}, \ldots, r_{3n}$.
Remark that in order to ensure $r_{3n}>0$, we need to require
$r_{n-1}+r_{2n}-mr_n>0$. Here, we require a stronger condition
\[r_{n-1}>mr_n\; \;\;\;\;\text{(i.e.,
the inequality \eqref{sec-eq0} in Theorem \ref{thm:mss})}.\]
In fact, if we have $r_{n-1}>mr_n$, then for any $r_{2n}>0$, we can make it sure $r_{3n}>0$.
We make this stronger requirement on
$r_{n-1}$ and $r_n$
because we need
more flexibility on $r_{2n}$ later when we prove Theorem \ref{thm:mss}.

\subsection{Nondegenerate multistationarity}\label{sec:mss}
In this subsection, we prove Theorem \ref{thm:mss}. We give an outline of the proof below.

\underline{First}, we consider a simpler network.  For $i\neq 2$, we remove the inflow reactions $X_i\rightarrow 0$ from $\widetilde{K}_{m,n}$ and obtain a subnetwork. 
Notice that for the subnetwork, we have $r_{n+1}=r_{n+3}=\cdots=r_{2n}=0$ in $f$ \eqref{eq:system} on the right-hand side of mass-action ODEs.
We  show  in Lemma \ref{lm} that for this special choice of rate constants,  the system $f=0$ has two
nondegenerate positive solutions under conditions \eqref{eq:nondegenerate} and \eqref{eq:conie-eq11}--\eqref{eq:conie-eq12}. In order to prove Lemma \ref{lm}, we need two results from linear algebra; see Lemmas \ref{sec4-lm1} and \ref{lm:J0}.

\underline{Second}, for $r_{n+1}=r_{n+3}=\cdots=r_{2n}=0$, besides two solutions $x^{(1)}$ and $x^{(2)}$ shown in Lemma \ref{lm},  the system $f=0$ \eqref{eq:system} has a ``special" solution $x^{(3)}$  with
its last coordinate $x_{n}=+\infty$.
We make this third solution  ``visible" by applying a variable substitution to the system $f$ (see  \eqref{eq:defp}--\eqref{eq:defg}). Equivalently, we show the resulting system $g$ in \eqref{eq:defg} has a nondegenerate positive solution (under the condition \eqref{eq:conie-eq11} for $n=3$, or the conditions \eqref{eq:conie-eq12}--\eqref{eq:sec3-eq12} for $n>3$), which gives the third solution $x^{(3)}$  to the original system $f=0$;
see Lemma \ref{lm:thirdn3} for $n=3$ and Lemma \ref{lm:third} for $n>3$.

\underline{Finally}, we set $r_{n+1}=r_{n+3}=\cdots=r_{2n}=\epsilon$. By the previous steps and the implicit function theorem, we  show that
$f=0$ has three nondegenerate positive solutions if $\epsilon$ is a sufficiently small positive number; see Lemma \ref{lm:jac2} and the proof of Theorem \ref{thm:mss}.

\begin{lemma}\label{sec4-lm1}
For any $n\geq 3$, the determinant of the tridiagonal matrix
\begin{equation*}
\begin{bmatrix}
a_1+b_1&a_2&&&\\
b_1&a_2+b_2&a_3&&\\
&\ddots&\ddots&\ddots&\\
&&b_{n-2}&a_{n-1}+b_{n-1}&a_n\\
&&&b_{n-1}&a_n
\end{bmatrix}
\end{equation*}
is equal to $a_1a_2\ldots a_n$.
\end{lemma}
\begin{proof}
We transform the matrix into an upper triangular matrix by
applying  the Gaussian elimination  starting from the last row to the first row:
\begin{equation*}
\begin{bmatrix}
a_1+b_1&a_2&&&\\
b_1&a_2+b_2&a_3&&\\
&\ddots&\ddots&\ddots&\\
&&b_{n-2}&a_{n-1}+b_{n-1}&a_n\\
&&&b_{n-1}&a_n
\end{bmatrix}
\end{equation*}
\begin{equation*}
\longrightarrow\begin{bmatrix}
a_1+b_1&a_2&&&\\
b_1&a_2+b_2&a_3&&\\
&\ddots&\ddots&\ddots&\\
&&b_{n-2}&a_{n-1}&0\\
&&&b_{n-1}&a_n
\end{bmatrix}
\end{equation*}
\begin{equation*}
\longrightarrow\cdots\longrightarrow\begin{bmatrix}
a_1&0&&&\\
b_1&a_2&0&&\\
&\ddots&\ddots&\ddots&\\
&&b_{n-2}&a_{n-1}&0\\
&&&b_{n-1}&a_n
\end{bmatrix}.
\end{equation*}
Thus, the determinant is $a_1a_2\ldots a_n$.
\end{proof}

\begin{lemma}\label{lm:J0}
For any  integer $m\geq 2$, and for any odd integer $n\geq 3$, 
if
\begin{equation}\label{eq:con0}
r_{n+1}=0\; \text{and}\; \;r_{n+i}=0, \;\; \text{for }\; 3\leq i \leq n,
\end{equation}
then the determinant of $J$ in \eqref{eq:jac} is
\begin{align*}
r_2\cdots r_{n-1} x_2\cdots x_{n-1}((m-1)r_1r_nx_1-(r_n+r_1x_2)r_{n+2}).
\end{align*}
\end{lemma}

\begin{proof}
We expand $\det J|_{r_{n+1}=0\; \text{and}\; r_{n+i}=0, \;3\leq i \leq n}$ with respect to the first row and obtain $\det J=-(r_1x_2+r_n)\det J_1+r_1x_1\det J_2$, where
\begin{equation*}
J_1~=~
{\fontsize{3pt}{2pt}
\begin{bmatrix}
-r_1x_1-r_2x_3-r_{n+2}&\cdots&0&0\\
-r_2x_3&\ddots&\vdots&\vdots\\
0&\ddots&-r_{n-2}x_{n-2}&0\\
\vdots&\ddots&-r_{n-2}x_{n-2}-r_{n-1}x_n&-r_{n-1}x_{n-1}\\
0&\cdots&-r_{n-1}x_n&-r_{n-1}x_{n-1}
\end{bmatrix}
}
\end{equation*}
and
\begin{equation*}
J_2~=~
{\fontsize{3pt}{2pt}
\begin{bmatrix}
-r_1x_2&-r_2x_2&\cdots&0\\
0&-r_2x_2-r_3x_4&\ddots&\vdots\\
\vdots&-r_3x_4&\ddots&0\\
0&\ddots&-r_{n-2}x_{n-2}-r_{n-1}x_n&-r_{n-1}x_{n-1}\\
mr_n&\cdots&-r_{n-1}x_n&-r_{n-1}x_{n-1}
\end{bmatrix}.
}
\end{equation*}
By Lemma \ref{sec4-lm1}, $\det J_1=r_2x_2\cdots r_{n-1}x_{n-1}(r_1x_1+r_{n+2})$. Again, we expand $\det J_2$ with respect to the first column: $\det J_2=-r_1x_2\det J_3-mr_n\det J_4$, where
\begin{equation*}
J_3=
{\fontsize{3pt}{2pt}
\begin{bmatrix}
-r_2x_2-r_3x_4&-r_3x_3&\cdots&0\\
-r_3x_4&-r_3x_3-r_4x_5&\ddots&\vdots\\
0&\ddots&\ddots&0\\
\vdots&\ddots&-r_{n-2}x_{n-2}-r_{n-1}x_n&-r_{n-1}x_{n-1}\\
0&\cdots&-r_{n-1}x_n&-r_{n-1}x_{n-1}
\end{bmatrix}
}
\end{equation*}
and
\begin{equation*}
J_4=
{\fontsize{3pt}{2pt}
\begin{bmatrix}
-r_2x_2&0&\cdots&0\\
-r_2x_2-r_3x_4&-r_3x_3&\ddots&\vdots\\
-r_3x_4&\ddots&0&0\\
\vdots&\ddots&-r_{n-2}x_{n-2}&0\\
0&\cdots&-r_{n-2}x_{n-2}-r_{n-1}x_n&-r_{n-1}x_{n-1}
\end{bmatrix}
}.
\end{equation*}
By Lemma \ref{sec4-lm1},  $\det J_3=-r_2x_2\cdots r_{n-1}x_{n-1}$. Clearly, $\det J_4=-r_2x_2\cdots r_{n-1}x_{n-1}$. Thus
\begin{align*}
\det J&=-(r_1x_2+r_n)r_2x_2\ldots r_{n-1}x_{n-1}(r_1x_1+r_{n+2})\\
&\quad\,+r_1x_1(r_1x_2+mr_n)r_2x_2\ldots r_{n-1}x_{n-1}\\
&=r_2\cdots r_{n-1} x_2\cdots x_{n-1}((m-1)r_1r_nx_1-(r_n+r_1x_2)r_{n+2}).
\end{align*}
\end{proof}

\begin{lemma}\label{lm}
For any  integer $m\geq 2$,  and for any odd integer $n\geq 3$,  
if the  rate constants $r_{n+1}, r_{n+3}, \ldots, r_{2n}$ satisfy the condition \eqref{eq:con0},
 and if the positive rate constants $r_{1}, \ldots, r_{n}, r_{n+2}$  satisfy the inequalities \eqref{eq:nondegenerate} and \eqref{eq:conie-eq11}--\eqref{eq:conie-eq12},
then the system $f$ in \eqref{eq:system} has two distinct positive solutions
\begin{align}\label{eq:2solutions}
x^{(1)}=(1,1,\ldots,1)  \;\;\;
\text{and} \;\;\;  x^{(2)}=(\delta_1,\delta_2,\ldots,\delta_n),
\end{align}
where
$$\delta_1~:=~\frac{(r_1+r_n)r_{n+2}}{(m-1)r_1r_n},\;\;\;\delta_2~:=~\frac{((m-1)r_1-r_{n+2})r_n}{r_1r_{n+2}},$$
and
$$\delta_i~:=~\frac{(m-1)r_1(r_{i-1}+(-1)^{i}mr_{n})+(-1)^{i-1}m(r_1+r_{n})r_{n+2}}{(m-1)r_1r_{i-1}\delta_{i-1}},\quad i=3,\ldots,n,$$
and the Jacobian matrix $J$ in \eqref{eq:jac} has full rank at both
solutions.
\end{lemma}
\begin{proof}
First, it is straightforward to check that if the rate constants satisfy condition \eqref{eq:coneq}, then $x^{(1)}=(1,1,\ldots,1)$ is a positive solution
to $f(x)=0$ for $f$ in \eqref{eq:system}.

Below, we show how to obtain the other solution $x^{(2)}$ to $f(x)=0$. Note that  under  the condition \eqref{eq:con0}, we have
\[\sum_{j=1}^n(-1)^{j-1}f_j=(m-1)r_nx_1+r_{n+2}x_2-(m-1)r_n-r_{n+2}.\]
We solve for $x_2$ from $\sum_{j=1}^n(-1)^{j-1}f_j=0$, substitute the expression into $f_1=0$, and obtain a quadratic equation in terms of only $x_1$:
\[(m-1)r_1r_nx_1^2-((m-1)r_1r_n+(r_1+r_n)r_{n+2})x_1+(r_1+r_n)r_{n+2}=0,\]
which indeed has two solutions: $x^{(1)}_1=1$ and $x^{(2)}_1=\delta_1$. We substitute $x^{(2)}_1=\delta_1$ into $f_1=0$ and solve that $x^{(2)}_2=\delta_2$. Note that for $3\le i\le n$,  under  the condition \eqref{eq:con0},  we have
\[\sum_{j=1}^{i-1}(-1)^{j-1}f_j=(-1)^{i-1}r_{i-1}x_{i-1}x_{i}+(-1)^{i}r_{i-1}-r_n(x_1-1)+r_{n+2}(x_2-1).\]
So, we can substitute $x^{(2)}_1=\delta_1$ and $x^{(2)}_2=\delta_2$ into $\sum_{j=1}^{i-1}(-1)^{j-1}f_j=0$ and  solve $x^{(2)}_i=\delta_i$ for $i=3,\ldots,n$.
Hence $x^{(2)}=(\delta_1,\delta_2,\ldots,\delta_n)$ is also a solution
to $f(x)=0$  under  the condition \eqref{eq:con0}. If the rate constants satisfy conditions \eqref{eq:conie-eq11}--\eqref{eq:conie-eq12}, then $x^{(2)}$ is clearly positive.
Obviously, under the condition \eqref{eq:nondegenerate}, we have $1\neq \delta_1$, and hence $x^{(1)}\neq x^{(2)}$.

Below, we show that if
the inequality \eqref{eq:nondegenerate} is satisfied, then at both solutions $x^{(1)}$ and $x^{(2)}$, the Jacobian matrix $J$ has nonzero determinants.
In fact,  by Lemma \ref{lm:J0},
\begin{align*}
\det J ~=~ r_2\cdots r_{n-1} x_2\cdots x_{n-1}((m-1)r_1r_nx_1-(r_n+r_1x_2)r_{n+2}).
\end{align*}
Therefore,
$$\det J|_{x=x^{(1)}}=r_2\cdots r_{n-1}((m-1)r_1r_n-(r_1+r_n)r_{n+2})$$
and
$$\det J|_{x=x^{(2)}}=r_2\cdots r_{n-1} \delta_2\cdots \delta_{n-1}((r_1+r_n)r_{n+2}-(m-1)r_1r_n),$$
which are nonzero if \eqref{eq:nondegenerate} holds.
\end{proof}

As mentioned before, for the special choice of rate-constant values in the condition \eqref{eq:con0}, besides two solutions $x^{(1)}$ and $x^{(2)}$ shown in Lemma \ref{lm},  the polynomial system $f=0$ in \eqref{eq:system} has a ``special" solution  with
its last coordinate $x_{n}=+\infty$. In order to make this third solution ``visible", we need to apply a variable substitution to the system $f$.

First,  define a map $\varphi: {\mathbb R}^{n}\rightarrow {\mathbb R}^n$ as follows:
\begin{align}\label{eq:mapv}
\varphi(y_1, \ldots, y_n)~=~\left(y_1, \ldots, y_{n-2}, \frac{r_{2n}y_{n-1}}{y_n}, \frac{y_n}{r_{2n}}
\right).
\end{align}
We substitute $x=\varphi(y)$ into $f$ in (\ref{eq:system}) and view $y_1,\ldots,y_n$ as new variables. We define the resulting rational functions as
\begin{align}\label{eq:defp}
p(y_1, \ldots, y_n;r_1, \ldots, r_{2n}):= f|_{x=\varphi(y)} \;\in\; {\mathbb Q}(r_1, \ldots, r_{2n}, y_1, \ldots, y_n).
\end{align}
Then, substitute \eqref{eq:con0} into  
 $p$, and define the resulting polynomials as
 \begin{align}\label{eq:defg}
 g(y_1, \ldots, y_n;r_1, \ldots, r_n, r_{n+2}):=p|_{r_{n+1}=0, \;r_{n+i}=0\; (3\leq i\leq n)}.
 \end{align}
 Denote by $J_{p}$ and $J_g$ respectively the Jacobian matrix of $p$ and $g$ with respect to variables $y_1, \ldots, y_{n}$.

 When $n=3$, the system $g$ in \eqref{eq:defg} is given by the polynomials:
\begin{equation}
\begin{cases}
g_1 ~=~ r_1+r_3-r_3y_1,\\
g_2 ~=~ r_1+r_2+r_{5}-r_2y_2, \\
g_3 ~=~ r_{2}-mr_3-r_{2}y_{2}+mr_{3}y_1-y_3.
\end{cases}\label{eq:system3n3}
\end{equation}


\begin{lemma}\label{lm:thirdn3}
 For any  integer $m\geq 2$, if the positive rate constants $r_1$ and $r_5$ satisfy, 
\begin{align}\label{eq:condn3}
(m-1)r_1-r_5~>~0,
\end{align}
then  for any positive rate constant $r_2$ and $r_3$, the system $g=0$ in \eqref{eq:system3n3} has a positive solution $\xi=(\xi_1, \xi_2, \xi_3)$ such that $\det J_g|_{y=\xi}\neq 0$.
\end{lemma}

\begin{proof}
Solve the system $g=0$ from (\ref{eq:system3n3})  for the variables $y_1,y_2,y_3$ over ${\mathbb Q}(r)$, and obtain a solution in terms of $r$:
$$\xi_1~:=~\frac{r_1+r_3}{r_3},\;\; \xi_2~:=~\frac{r_1+r_2+r_5}{r_2},\; \;\xi_3~:=~(m-1)r_1-r_5.$$
Clearly,  if $(m-1)r_1-r_5>0$, then for any positive $r_2$ and $r_3$, the above solution is positive.
It is straightforward to compute that $\det J_g|_{y=\xi}=-r_2r_3\neq 0$.
\end{proof}

\begin{remark}
Remark that the inequality \eqref{eq:condn3} is a specific case of  the inequality \eqref{eq:conie-eq11} for $n=3$.
\end{remark}

 Now, we focus on the case when $n\geq 5$.
 Explicitly, the form of $p$ in \eqref{eq:defp} for $n\geq 5$ is given below:
{\footnotesize
\begin{equation*}
\begin{cases}
p_1  ~=~ -r_1y_1y_2-r_ny_1-r_{n+1}y_1+r_1+r_n,\\
p_2 ~=~  -r_1x_1x_2-r_2x_2x_3 - r_{n+2}x_2+r_1+r_2+r_{n+2},\\
p_i ~=~ -r_{i-1}y_{i-1}y_i-r_iy_iy_{i+1}-r_{n+i}y_i+ r_{i-1}+r_i + r_{n+i}, \quad\quad \textrm{ for }2\le i\le n-3,\\
p_{n-2}~=~ -r_{n-3}y_{n-3}y_{n-2}-r_{n-2}y_{n-2}\frac{r_{2n}y_{n-1}}{y_n}-r_{2n-2}y_{n-2}+r_{n-3}+r_{n-2}+r_{2n-2}, \\
p_{n-1}~=~ -r_{n-2}y_{n-2}\frac{r_{2n}y_{n-1}}{y_n}-r_{n-1}y_{n-1}-r_{2n-1}y_{n-1}+r_{n-2}+r_{n-1} +r_{2n-1},\\
p_{n}~=~ -r_{n-1}y_{n-1}+mr_{n}x_1-y_n +r_{n-1}+r_{2n}-mr_n.
\end{cases}\label{eq:systemp}
\end{equation*}
}
 Explicitly, the form of $g$ in \eqref{eq:defg}  for $n\geq 5$ is given below:
\begin{equation}
{\footnotesize
\begin{cases}
g_1  ~=~ r_1+r_n-r_ny_1-r_1y_1y_2,\\
g_2 ~=~  r_1+r_2+r_{n+2}-r_1y_1y_2-r_2y_2y_3 - r_{n+2}y_2,\\
g_i ~=~ r_{i-1}+r_i-r_{i-1}y_{i-1}y_i-r_iy_iy_{i+1},\quad\quad\quad \;\;\;\textrm{ for }3\le i\le n-3,\\
g_{n-2}~=~ r_{n-3}+r_{n-2}-r_{n-3}y_{n-3}y_{n-2}, \\
g_{n-1}~=~ r_{n-2}+r_{n-1}-r_{n-1}y_{n-1},\\
g_{n}~=~ r_{n-1}-mr_n+mr_{n}x_1-r_{n-1}y_{n-1}-y_n.
\end{cases}\label{system1}
}
\end{equation}

\begin{lemma}\label{lm:third}
For any integer $m\geq 2$, and for any odd integer $n>3$, if the positive rate constants $r_1, \ldots, r_n, r_{n+2}$ satisfy the  in equalities \eqref{eq:sec3-eq11}--\eqref{eq:sec3-eq12} and
\begin{align}\label{eq:sec3-eq13}
(m-1)r_1r_{n-2}+m(m-1)r_1r_n~&>~m(r_1+r_n)r_{n+2},
\end{align}
 then
the system $g=0$ in \eqref{system1} has a positive solution $\xi=(\xi_1, \ldots, \xi_n)$ such that $\det J_g|_{y=\xi}\neq 0$.
\end{lemma}

\begin{proof}
The goal is to find a positive solution $\xi=(\xi_1, \ldots, \xi_n)$ to the equations $$g_1(y;r_1, \ldots, r_n, r_{n+2})=\cdots=g_n(y;r_1, \ldots, r_n, r_{n+2})=0$$ for positive parameter values $r_1, \ldots, r_n, r_{n+2}$.
First, 
we solve for $y_{n-1}$ from $g_{n-1}=0$ over ${\mathbb Q}(r)$, and we get
\begin{align}\label{eq:elimyn1}
y_{n-1}~=~\frac{r_{n-1}+ r_{n-2}}{r_{n-1}}.
\end{align}
Second, we substitute \eqref{eq:elimyn1} into $g_n$, and then we solve for $y_n$ from  $g_{n}=0$ over ${\mathbb Q}(r, y_1)$:
\begin{align}\label{eq:elimyn}
y_{n}~=~mr_n(y_1-1)-r_{n-2}.
\end{align}
Now, we show how to solve  for $y_2,\ldots, y_{n-2}$ from (\ref{system1})  over ${\mathbb Q}(r, y_1)$. For this purpose,  for every $i=2, \ldots, n-2$, let
$h_{i} ~=~ \Sigma_{k=i}^{n-2}(-1)^kg_k$.
Notice that $n$ is odd. So, explicitly, we obtain
\begin{align*}\label{eq:elimh}
h_2~&=~r_1 + r_{n+2} -r_1y_1y_2 - r_{n+2}y_2 - r_{n-2}, \quad\quad \text{and}\\
h_i~&=~(-1)^i\left(r_{i-1}-r_{i-1}y_{i-1}y_i\right)-r_{n-2} \quad\quad \text{for}~i=3, \ldots, n-2.
\end{align*}
We solve for $y_i$ from $h_i=0$, and we have
\begin{align}
y_2~&=~\frac{r_1-r_{n-2}+ r_{n+2}}{r_1y_1 + r_{n+2}}, \quad\quad \text{and} \label{eq:elimx2}\\
y_i~&=~\frac{r_{i-1}-(-1)^{i}r_{n-2}}{r_{i-1}y_{i-1}}  \quad\quad \text{for}~i=3, \ldots, n-2.\notag
\end{align}
We substitute \eqref{eq:elimx2} into $g_1$, and we obtain a quadratic polynomial in $y_1$:
\begin{equation*}\label{eq:eq1}
h_1~:=~r_1r_ny_1^2+(r_1r_{n+2}+r_nr_{n+2}-r_1r_{n-2}-r_1r_{n})y_1-(r_1+r_{n})r_{n+2}.
\end{equation*}
It is straightforward to check by the discriminant and Vieta's formulas that  for any positive
parameters $r_1, r_{n-2}, r_n, r_{n+2}$, the quadratic equation
$h_1(y_1)=0$ has two real roots, and only one of these two roots is positive.
Let $\xi_1$ be this positive root. Substituting $\xi_1$ back to \eqref{eq:elimyn1}, \eqref{eq:elimyn}, and \eqref{eq:elimx2}, we obtain a solution $\xi=(\xi_1,\ldots,\xi_{n-1},\xi_n)$ of $g=0$ in (\ref{system1}), where
$$\xi_2=\frac{r_1+r_{n+2}-r_{n-2}}{r_1\xi_1+r_{n+2}}, \quad \xi_i=\frac{r_{i-1}-(-1)^{i}r_{n-2}}{r_{i-1}\xi_{i-1}},i=3,\ldots,n-2$$
and
$$\xi_{n-1}=\frac{r_{n-2}+r_{n-1}}{r_{n-1}}, \quad \xi_n=mr_n(\xi_1-1)-r_{n-2}.$$

We show the positivity of this solution. Clearly, if \eqref{eq:sec3-eq11}  holds, then $\xi_2>0$. Also, if \eqref{eq:sec3-eq12} holds, then for every $i=3,\ldots,n-2$, $\xi_i>0$ holds.
Note that $\xi_n>0$ if $\xi_1>\frac{r_{n-2}}{mr_n}+1$. Note also $\xi_1$ is the only positive root of
$h_1(y_1)=0$. So, if $h_1(\frac{r_{n-2}}{mr_n}+1)<0$, then $\xi_1>\frac{r_{n-2}}{mr_n}+1$.
Note 
$$h_1(\frac{r_{n-2}}{mr_n}+1)=-\frac{r_{n-2}}{m^2r_n}\left((m-1)r_1r_{n-2}+m\left((m-1)r_1r_n-(r_1+r_n)r_{n+2}\right)\right).$$
So  for positive $m, r_{n-2}$ and $r_n$, $h_1(\frac{r_{n-2}}{mr_n}+1)<0$  is equivalent to \eqref{eq:sec3-eq13}.
Thus, if \eqref{eq:sec3-eq11}--\eqref{eq:sec3-eq12} and \eqref{eq:sec3-eq13} are satified, then $\xi$ is positive.

Finally, we show $\det J_g|_{y=\xi}\neq 0$. In fact, the Jacobian matrix of $g_1, \ldots,  g_n$ with respect to $y_1, \ldots, y_n$ is
\begin{equation*}
J_g=
{\fontsize{3pt}{2pt}
\begin{bmatrix}
-r_1y_2-r_n&-r_1y_1&\cdots&0&0&0\\
-r_1y_2&-r_1y_1-r_2y_3-r_{n+2}&\ddots&\vdots&\vdots&\vdots\\
0&-r_2y_3&\ddots&-r_{n-3}y_{n-3}&0&0\\
\vdots&0&\ddots&-r_{n-3}y_{n-3}&0&0\\
0&\vdots&\ddots&0&-r_{n-1}&0\\
mr_n&0&\cdots&0&-r_{n-1}&-1
\end{bmatrix}.
}
\end{equation*}
Expanding $\det J_g$ with respect to the first row and taking advantage of Lemma \ref{sec4-lm1}, we have
$$\det J_g=-r_2\cdots r_{n-3}r_{n-1}y_2\cdots y_{n-3}(r_{n}r_{n+2}+r_1r_ny_1+r_1r_{n+2}y_2)$$
which is obviously nonzero at any positive solution $y=\xi$.
\end{proof}

\begin{remark}
Remark that the inequality \eqref{eq:sec3-eq13} is a specific case of  the inequality \eqref{eq:conie-eq12} for $i=n-1$.
\end{remark}

\begin{lemma}\label{lm:jac2}
If for a choice of the rate constants $r_1, \ldots, r_{2n}$,
 the system $p=0$ has a solution  $\xi=(\xi_1, \ldots,  \xi_{n})$  such that $\xi_n\neq 0$ and  $\det J_p|_{y=\xi}\neq 0$, and if the rate constant $r_{2n}\neq 0$, then for the same choice of rate constants,
$\hat x:=(\xi_1, \ldots,  \xi_{n-2}, \frac{r_{2n}\xi_{n-1}}{ \xi_n},\frac{ \xi_n}{r_{2n}})$ is  a solution to $f$ such that $\det J|_{x=\hat x}\neq 0$.
\end{lemma}
\begin{proof}
 By the definition of the map $\varphi$ \eqref{eq:mapv}, and by the definition of the system $p$ \eqref{eq:defp}, $\hat x$ is a solution to $f$.
Note also, the Jacobian matrix of $\varphi$ with respect to $y_1, \ldots, y_n$ is
\begin{equation*}
J_{\varphi}:=\begin{bmatrix}
{\mathrm I}_{n-2}&\\
&\begin{matrix}
\frac{y_n}{r_{2n}}&\frac{r_{2n}y_{n-1}}{y_{n}}\\
&r_{2n}
\end{matrix}
\end{bmatrix},
\end{equation*}
where ${\mathrm I}_{n-2}$ denotes the identity matrix of size $(n-2)\times (n-2)$.
So, if $\xi_n\neq 0$ and $r_{2n}\neq 0$, $J_{\varphi}|_{y=\xi}$ is invertible.
By \eqref{eq:defp}, we have
$J_{p}=J\cdot J_{\varphi}$. So we conclude that $\det J|_{x=\hat x}\neq 0$.
\end{proof}




{\bf Proof of Theorem \ref{thm:mss}.}
Let $r_{n+1}=\epsilon$, and for $i=3, \ldots n$, let $r_{n+i}=\epsilon$.
For $n=3$, choose values for the rate constants $r_1, r_2, r_3, r_5$ such that the conditions \eqref{eq:nondegenerate}--\eqref{eq:conie-eq12} are satisfied.  For instance,
we can choose
\begin{align}\label{eq:valuen3}
r_1=2, \;\; r_2=m+1, \;\; r_3=1, \;\; \text{and} \;\; r_5=m-1.
\end{align}
For any odd integer $n>3$, choose values for the rate constants $r_1, \ldots, r_n,  r_{n+2}$ such that the conditions \eqref{eq:nondegenerate}--\eqref{eq:sec3-eq12} are satisfied.
For instance, we can choose
\begin{align}\label{eq:value}\notag
&r_1=2,\quad r_2=r_4=\cdots=r_{n-3}=r_n=1, \\
&r_3=r_5=\cdots=r_{n-4}=r_{n-1}=m+1,\quad r_{n-2}=m, \quad r_{n+2}=m-1.
\end{align}
Here, by the values in \eqref{eq:valuen3} (or, the values in \eqref{eq:value}), we see the open set determined by the inequalities \eqref{eq:nondegenerate}--\eqref{eq:conie-eq12}  for $n=3$  (or, the inequalities \eqref{eq:nondegenerate}--\eqref{eq:sec3-eq12} for $n>3$) is non-empty.
For $i=2n+1, \ldots, 3n$, set $r_{i}$ as \eqref{eq:coneq}.

By Lemma \ref{lm}, if $\epsilon=0$, then  $f=0$ in \eqref{eq:system} has two distinct positive solutions $x^{(1)}=(1, 1, \ldots, 1)$ and $x^{(2)}$ in \eqref{eq:2solutions}.
So, by the implicit function theorem, if $\epsilon$ is  a sufficiently small positive number, then $f=0$ has two distinct positive solutions $\hat x^{(1)}$ and $\hat x^{(2)}$ with $\det J|_{x=\hat x^{(i)}}\neq 0, (i=1,2)$, where $\hat x^{(1)}=x^{(1)}$ (since $x^{(1)}$ is always a solution to $f=0$ in \eqref{eq:system}), and $\hat x^{(2)}$ is sufficiently close to $x^{(2)}$.
That means $\widetilde{K}_{m,n}$ has at least two distinct nondegenerate steady states. 

 By Lemmas \ref{lm:thirdn3}--\ref{lm:third},
  for the rate constants \eqref{eq:valuen3} when $n=3$,
 or respectively   for the rate constants \eqref{eq:value} when $n>3$,
the system $g=0$ in  \eqref{eq:defg} has a positive solution $\xi=(\xi_1, \ldots, \xi_n)$ such that
$\det J_g|_{y=\xi}\neq 0$.   By the definition of $g$ in \eqref{eq:defg}, when $\epsilon=0$, $\xi$ is also a positive solution of the system $p=0$ in \eqref{eq:defp} such that
$\det J_p|_{y=\xi}\neq 0$ for the same choice of $r_1, \ldots, r_n,  r_{n+2}$.
Therefore, by the implicit function theorem, if $\epsilon$ is a sufficiently small positive number, then $p=0$ has a positive solution, say $\hat\xi=(\hat\xi_1, \ldots, \hat\xi_n)$, which is close to $\xi$, such that
$\det J_p|_{y=\hat\xi}\neq 0$. Let $\hat x^{(3)}=(\hat \xi_1, \ldots, \hat \xi_{n-2}, \frac{\epsilon \hat \xi_{n-1}}{\hat \xi_n},\frac{\hat \xi_n}{\epsilon})$. By Lemma \ref{lm:jac2}, $\hat x^{(3)}$ is a positive solution to the system $f=0$ such that $\det J|_{x=\hat x^{(3)}}\neq 0$. So $\hat x^{(3)}$ is the third nondegenerate steady state of $\widetilde{K}_{m,n}$.

\subsection{Bistability}\label{sec:proofbi}
Here, we prove that two of those
three steady states stated in Theorem \ref{thm:mss} are stable if we replace the condition \eqref{eq:nondegenerate} in Theorem \ref{thm:mss} with the condition \eqref{eq:constable1} (see Theorem \ref{thm:bistability}).  The main idea is to show
the Jacobian matrices at two steady states are similar to column diagonally dominant matrices (see Lemmas \ref{lm:1stable}-\ref{lm:similar5}). Then, we can conclude
bistability by Theorem \ref{thm:matrix2}.

\begin{lemma}\label{lm:1stable}
For any integer $m\geq 2$,  and for any odd integer $n\geq 3$, 
if the  rate constants $r_{n+1}, r_{n+3}, \ldots, r_{2n}$ satisfy the condition \eqref{eq:con0}, and
if the rate constants $r_1, r_n, r_{n+2}$ satisfy the inequality \eqref{eq:constable1},
then for $x^{(1)}=(1,1,\ldots,1)$, the matrix $J|_{x=x^{(1)}}$ is similar to a column diagonally dominant matrix.
\end{lemma}

\begin{proof}
Since the condition  \eqref{eq:con0} holds, for any $r_1, \ldots, r_n, r_{n+2}$,
the Jacobian matrix $J$ is as follows:
\begin{equation*}
{\fontsize{3pt}{2pt}J~=~\begin{bmatrix}
-r_1x_2-r_n&-r_1x_1&\cdots&0&0\\
-r_1x_2&-r_1x_1-r_2x_3-r_{n+2}&\cdots&\vdots&\vdots\\
0&-r_2x_3&\ddots&0&0\\
\vdots&0&\ddots&-r_{n-2}x_{n-2}&0\\
0&\vdots&\ddots&-r_{n-2}x_{n-2}-r_{n-1}x_n&-r_{n-1}x_{n-1}\\
mr_n&0&\cdots&-r_{n-1}x_n&-r_{n-1}x_{n-1}
\end{bmatrix}}.
\end{equation*}
Let $\alpha=1+\frac{r_{n+2}}{r_1x_1}$, and
let $D$ be the diagonal matrix $\diag(\alpha,1,\ldots,1)$. Note that the matrix $\tilde J:=DJD^{-1}$ is equal to
\begin{equation*}
{\fontsize{3pt}{2pt}
\tilde J ~=~
\begin{bmatrix}
-r_1x_2-r_n&-r_1x_1\alpha&\cdots&0&0\\
-\frac{r_1x_2}{\alpha}&-r_1x_1-r_2x_3-r_{n+2}&\cdots&\vdots&\vdots\\
0&-r_2x_3&\ddots&0&0\\
\vdots&0&\ddots&-r_{n-2}x_{n-2}&0\\
0&\vdots&\ddots&-r_{n-2}x_{n-2}-r_{n-1}x_n&-r_{n-1}x_{n-1}\\
\frac{mr_n}{\alpha}&0&\cdots&-r_{n-1}x_n&-r_{n-1}x_{n-1}
\end{bmatrix}}.
\end{equation*}
We denote by $a_{ij}$ the $(i, j)$-entry in $\tilde J$.
Clearly, for $i>2$, we have $|a_{ii}|=\sum_{j\neq i}|a_{ij}|$.
For $i=2$, by $\alpha=1+\frac{r_{n+2}}{r_1x_1}$, we have
\[|a_{22}|~=~r_1x_1+r_2x_3+r_{n+2}~=~\alpha r_1x_1 + r_2x_3 ~=~ \sum_{j\neq 2}|a_{2j}|.\]
Note that the inequality $|a_{11}|>\sum_{j\neq 1}|a_{1j}|$ is equivalent to
\begin{equation}\label{sec4-eq0}
\frac{(r_1x_2+r_n)r_{n+2}}{x_1}~>~(m-1)r_1r_n.
\end{equation}
For $x=x^{(1)}=(1,1,\ldots,1)$, the inequality \eqref{sec4-eq0} is exactly the inequality \eqref{eq:constable1}.
\end{proof}

\begin{remark}
Similarly to Lemma \ref{lm:1stable}, one can prove
if the rate constants satisfy the inequality \eqref{eq:constable2},
then for $x^{(2)}=(\delta_1,\delta_2,\ldots,\delta_n)$ \eqref{eq:2solutions} stated in Lemma \ref{lm}, $J|_{x=x^{(2)}}$ is similar to a column diagonally dominant matrix.
\end{remark}

\begin{lemma}\label{lm:similar3}
For any  integer $m\geq 2$, and for $n=3$, 
if $r_4=r_6=\epsilon$, and if  $\epsilon>0$ is sufficiently small,
then for any positive rate constants $r_1, r_2, r_3, r_5$, and for any positive numbers $\hat \xi_1, \hat \xi_2, \hat \xi_3$,  the matrix $J|_{x=(\hat \xi_1,\frac{\epsilon \hat \xi_{2}}{\hat \xi_3},\frac{\hat \xi_3}{\epsilon})}$ is similar to a column diagonally dominant matrix.

\end{lemma}
\begin{proof}
Let $D=\diag(d_1,1,d_3)$, where $d_1$ and $d_3$ satisfy the equalities:
\begin{equation}\label{sec4-eq4}
\begin{cases}
d_1r_{1}x_{1}+d_3r_{2}x_3~=~r_{1}x_{1}+r_{2}x_3+r_{5}\\
\frac{1}{d_3}r_{2}x_{2}~=~r_{2}x_{2}+r_{6}
\end{cases}.
\end{equation}
We solve for $d_1$ and $d_3$  from \eqref{sec4-eq4} over ${\mathbb Q}(r, x)$:
\begin{align}\label{eq:dpp3}
\begin{cases}
d_1 ~=~ \frac{r_1r_2x_1x_2+r_1r_6x_1+r_2r_5x_2+r_2r_6x_3+r_5r_6}{r_1x_1(r_2x_2+r_6)},  \\
d_3 ~=~ \frac{r_2x_2}{r_2x_2+r_6}.
\end{cases}
\end{align}
Notice that the matrix $\tilde J~:=~DJD^{-1}$ is equal to
\begin{equation*}
\begin{bmatrix}
\begin{matrix}
-r_1x_2-r_3-r_{4}&-d_1r_1x_1&0\\
-\frac{1}{d_1}r_1x_2&-r_1x_1-r_2x_3-r_{5}&-\frac{1}{d_3}r_2x_2\\
\frac{d_3}{d_1}mr_3 & -d_3r_2x_3 & -r_2x_2-r_6
\end{matrix}
\end{bmatrix}.
\end{equation*}
We denote by $a_{ij}$ the $(i, j)$-entry in $\tilde J$.
By \eqref{sec4-eq4},  for $i=2, 3$, we have $|a_{ii}|=\sum_{j\neq i}|a_{ij}|$.
By Definition \ref{def:dom},  in order to make $\tilde J$ to be column diagonally dominant, we only need to ensure $|a_{11}|\leq \sum_{j\neq 1}|a_{1j}|$. That means, it is sufficient to show that
for $x=(\hat \xi_1,\frac{\epsilon \hat \xi_{2}}{\hat \xi_3},\frac{\hat \xi_3}{\epsilon})$, and
for $r_4=r_6=\epsilon$, the inequality below is true if $\epsilon>0$ is sufficiently small:
\begin{equation}\label{sec4-eq5}
\frac{d_3}{d_{1}}mr_{3}+\frac{1}{d_{1}}r_{1}x_2~\le ~r_{1}x_{2}+r_{3}+r_{4}.
\end{equation}
In fact, we substitute \eqref{eq:dpp3} into the inequality \eqref{sec4-eq5} and obtain
\begin{equation}\label{sec4-eq6}
r_1x_1x_2(mr_2r_3+r_1(r_2x_2+r_6))\le(r_2r_6x_3+(r_1x_1+r_5)(r_2x_2+r_6))(r_1x_2+r_3+r_4).
\end{equation}
When $x=(\hat \xi_1,\frac{\epsilon \hat \xi_{2}}{\hat \xi_3},\frac{\hat \xi_3}{\epsilon})$ and $r_{4}=r_{6}=\epsilon$,  the inequality \eqref{sec4-eq6} is
\begin{equation}\label{sec4-eq7}
r_1\hat \xi_1\frac{\hat \xi_{2}\epsilon}{\hat\xi_3}(mr_2r_3+r_1r_2\frac{\hat\xi_{2}\epsilon}{\hat\xi_3}+\epsilon)\le(r_2\hat \xi_3+(r_1\hat \xi_1+r_5)(r_2\frac{\hat \xi_{2}\epsilon}{\hat\xi_3}+\epsilon))(r_3+r_1\frac{\hat \xi_{2}\epsilon}{\hat\xi_3}+\epsilon).
\end{equation}
Note that both sides of \eqref{sec4-eq7} are quadratic functions in $\epsilon$. Note also,  at $\epsilon=0$, the function on the left-hand side evaluates to $0$, while  the one on the right-hand side is positive. So for sufficiently small $\epsilon>0$,
the inequality \eqref{sec4-eq7} holds.
\end{proof}

\begin{lemma}\label{lm:similar5}
For any  integer $m\geq 2$,  and for  any odd integer $n>3$,
if $r_{n+1}=r_{n+3}=\ldots=r_{2n}=\epsilon$, and
if  $\epsilon>0$ is sufficiently small,
then for any positive rate constants $r_1, \ldots, r_{n}, r_{n+2}$,  and for any positive numbers $\hat \xi_1, \ldots, \hat \xi_n$, the matrix
$$J|_{x=(\hat \xi_1, \ldots, \hat \xi_{n-2}, \frac{\epsilon \hat \xi_{n-1}}{\hat \xi_n},
 \frac{\hat \xi_n}{\epsilon})}$$ is similar to a column diagonally dominant matrix.
\end{lemma}
\begin{proof}
Let $D=\diag(d_1,1,\ldots,1,d_{n-1},d_n)$,  where $d_1,d_{n-1}$, and $d_n$ satisfy the equalities:
\begin{equation}\label{sec4-eq}
\begin{cases}
d_1r_1x_1=r_1x_1+r_{n+2}\\
\frac{1}{d_{n-1}}r_{n-2}x_{n-2}+\frac{d_n}{d_{n-1}}r_{n-1}x_n=r_{n-2}x_{n-2}+r_{n-1}x_n+r_{2n-1}\\
\frac{d_{n-1}}{d_n}r_{n-1}x_{n-1}=r_{n-1}x_{n-1}+r_{2n}
\end{cases}.
\end{equation}
We solve for $d_1,d_{n-1}$, and $d_n$ from \eqref{sec4-eq} over ${\mathbb Q}(r, x)$:
\begin{align}\label{eq:dpp5}
{\footnotesize
\begin{cases}
 d_1 ~ =~ \frac{r_1x_1+r_{n+1}}{r_1x_1}, \\
d_{n-1} ~=~ \frac{r_{n-1}r_{n-2}x_{n-1}x_{n-2}+r_{2n}r_{n-2}x_{n-2}}{r_{n-1}r_{n-2}x_{n-1}x_{n-2}+r_{2n}r_{n-2}x_{n-2}+r_{2n}r_{n-1}x_n+r_{2n-1}r_{n-1}x_{n-1}+r_{2n}r_{2n-1}}, \\
 d_n ~=~ \frac{r_{n-1}r_{n-2}x_{n-1}x_{n-2}}{r_{n-1}r_{n-2}x_{n-1}x_{n-2}+r_{2n}r_{n-1}x_n+r_{2n}r_{n-2}x_{n-2}+r_{2n-1}r_{n-1}x_{n-1}+r_{2n}r_{2n-1}}.
 \end{cases}
 }
\end{align}
 Notice that $\tilde J:= DJD^{-1}$ is equal to
 \begin{equation*}
{\tiny
\begin{bmatrix}
-r_1x_2-r_n-r_{n+1}&-d_1r_1x_1&\cdots&0&0\\
-\frac{1}{d_1}r_1x_2&-r_1x_1-r_2x_3-r_{n+2}&\cdots&\vdots&\vdots\\
0&-r_2x_3&\ddots&0&0\\
\vdots&0&\ddots&-\frac{1}{d_{n-1}}r_{n-2}x_{n-2}&0\\
0&\vdots&\ddots&-r_{n-2}x_{n-2}-r_{n-1}x_n-r_{2n-1}&-\frac{d_{n-1}}{d_n}r_{n-1}x_{n-1}\\
-\frac{d_n}{d_1}mr_n&0&\cdots&-\frac{d_n}{d_{n-1}}r_{n-1}x_n&-r_{n-1}x_{n-1}-r_{2n}
\end{bmatrix}.
}
\end{equation*}
We denote by $a_{ij}$ the $(i, j)$-entry in $\tilde J$.
Clearly, for any $2<i<n-2$, $|a_{ii}|=\sum_{j\neq i}|a_{ij}|$.
By \eqref{sec4-eq},  for $i=2, n-1, n$, we have $|a_{ii}|=\sum_{j\neq i}|a_{ij}|$.
 By Definition \ref{def:dom},  in order to make $\tilde J$ to be column diagonally dominant, we only need to make sure
 $|a_{ii}|\leq\sum_{j\neq i}|a_{ij}|$ for $i=1$ and $i=n-2$. That means that it is sufficient to show that for $x=(\hat \xi_1, \ldots, \hat \xi_{n-2}, \frac{\epsilon \hat \xi_{n-1}}{\hat \xi_n},\frac{\hat \xi_n}{\epsilon})$, and  for $r_{n+1}=r_{n+3}=\cdots=r_{2n}=\epsilon$,
 we have the inequalities below if $\epsilon>0$ is sufficiently small:
\begin{equation}\label{sec4-eq1}
\begin{cases}
\frac{1}{d_{1}}r_{1}x_{2}+\frac{d_n}{d_{1}}mr_{n}~\le~ r_{1}x_{2}+r_{n}+r_{n+1} \\
r_{n-3}x_{n-3}+d_{n-1}r_{n-2}x_{n-1}~\le~ r_{n-3}x_{n-3}+r_{n-2}x_{n-1}+r_{2n-2}
\end{cases}.
\end{equation}
By \eqref{eq:dpp5},  $d_{n-1}<1$. So, the second inequality in \eqref{sec4-eq1} holds for any positive $r$ and $x$.
 We substitute \eqref{eq:dpp5} into the first inequality in \eqref{sec4-eq1}. Then we have
\begin{equation}\label{sec4-eq2}
\begin{split}
mr_1r_{n-1}r_{n-2}r_{n}x_1x_{n-2}x_{n-1}~\le~&(r_1r_nx_1+r_1r_{n+1}x_1+r_{1}r_{n+2}x_{2}+r_{n}r_{n+2}\\
&+r_{n+1}r_{n+2})(r_{n-2}r_{n-1}x_{n-2}x_{n-1}+r_{n-1}r_{2n-1}x_{n-1}\\
&+r_{2n}(r_{n-2}x_{n-2}+r_{n-1}x_n+r_{2n-1})).
\end{split}
\end{equation}
For $x=(\hat \xi_1, \ldots, \hat \xi_{n-2}, \frac{\epsilon \hat \xi_{n-1}}{\hat \xi_n},\frac{\hat \xi_n}{\epsilon})$, and for $r_{n+1}=r_{n+3}=\cdots=r_{2n}=\epsilon$,  the inequality \eqref{sec4-eq2} is
\begin{equation}\label{sec4-eq3}
\begin{split}
(mr_1r_{n-1}r_{n-2}r_{n}\hat\xi_1\hat\xi_{n-2}\frac{\hat\xi_{n-1}}{\hat\xi_{n}})\epsilon~\le~&(r_1r_n\hat\xi_1+r_1r_{n+2}\hat\xi_2+r_{n}r_{n+2}+(r_{1}\hat\xi_1+r_{n+2})\epsilon)\\
&(r_{n-1}\hat\xi_{n}+(r_{n-1}\frac{\hat\xi_{n-1}}{\hat\xi_{n}}+1)(r_{n-2}\hat\xi_{n-2}\epsilon+\epsilon^2)).
\end{split}
\end{equation}
Note that when $\epsilon=0$, the left-hand side of \eqref{sec4-eq3} is zero, and the right-hand side is positive. So, the inequality \eqref{sec4-eq3} clearly holds for sufficiently small $\epsilon>0$.
\end{proof}

{\bf Proof of Theorem \ref{thm:bistability}.}
Let $r_{n+1}=r_{n+3}=\ldots=r_{2n}=\epsilon$.
For $n=3$, choose the rate constants $r_1, r_2, r_3, r_5$ as in \eqref{eq:valuen3}.
For $n>3$, choose the rate constants $r_1, \ldots, r_n,  r_{n+2}$ as in \eqref{eq:value}.
By the proof of Theorem \ref{thm:mss},  for these rate constants, $\widetilde{K}_{m,n}$ has three nondegenerate positive steady states $\hat x^{(i)}$ $(i=1, 2, 3)$ if $\epsilon$ is a sufficiently small positive number, where
$\hat x^{(3)}$ has the form
$(\hat \xi_1, \ldots, \hat \xi_{n-2}, \frac{\epsilon \hat \xi_{n-1}}{\hat \xi_n},\frac{\hat \xi_n}{\epsilon})$, and
$\hat x^{(1)}=(1,1,\ldots,1)$. By Lemmas \ref{lm:similar3}--\ref{lm:similar5} and Theorem \ref{thm:matrix2},  all non-zero eigenvalues of $J|_{x=\hat x^{(3)}}$ have negative real parts.
Note our choice of rate constants also satisfies the inequality \eqref{eq:constable1}. So, by Theorem \ref{thm:matrix2} and Lemma \ref{lm:1stable},
when $\epsilon=0$, all non-zero eigenvalues of $J|_{x=\hat x^{(1)}}$ have negative real parts.
Note that the eigenvalues of a matrix vary continuously under continuous perturbations of entries.
So,  if $\epsilon$ is a sufficiently small positive number, all non-zero eigenvalues of $J|_{x=\hat x^{(1)}}$ also have negative real parts. 
By the proof of Theorem \ref{thm:mss}, $\det J|_{x=\hat x^{(i)}}\neq 0, \;\text{for}\;i=1,3$. So both $\hat x^{(1)}$ and $\hat x^{(3)}$ are locally asymptotically stable.

\subsection{Non-empty open region for bistability}  {\bf Proof of Theorem \ref{thm:region}.}
The case for $n=3$ is obvious. For any odd integer $n>3$, by the inequality \eqref{eq:constable1}, the inequality \eqref{eq:conie-eq12} holds if and only if
\begin{equation}\label{eq:prop}
(m-1)r_1r_{i-1}~>~m((r_1+r_{n})r_{n+2}-(m-1)r_1r_{n}),\;\;i=4,\ldots,n-1.
\end{equation}
And by the inequalities \eqref{eq:sec3-eq12}, the inequality \eqref{eq:prop} holds if and only if
\begin{equation*}\label{eq2:prop}
(m-1)r_1r_{n-2}~>~m((r_1+r_{n})r_{n+2}-(m-1)r_1r_{n}).
\end{equation*}
Then, it is easy to see that the open set in ${\mathbb R}_{>0}^{n+1}$  determined by the inequalities \eqref{sec-eq0}--\eqref{eq:constable1} is equivalent to  the set given in \eqref{eq3:prop}.

\section{Summary}\label{con}
In this paper, we prove that the fully open extension of a sequestration network admits three nondegenerate positive steady states, two of which are locally asymptotically stable. The method we use to prove stability here is based on the Gershgorin circle theorem, which can be applied to more general chemical reaction networks. Moreover, we give an open region in the parameter space as well as explicit choices  of rate constants to ensure bistability. In the future, it would be interesting to invest the configuration of the positive steady states and study how the stability of them changes as the rate constants vary. Also, it is challenging to prove the maximum number of (stable) positive steady states for the fully open extension of a sequestration network (Conjecture \ref{conj}).

\section*{Acknowledgments}
We would like to acknowledge Anne Shiu for her generous support and valuable advice. We also thank Badal Joshi for his comments on the first draft of our work.

\end{document}